\documentclass[11pt]{amsart}  

\usepackage{amsmath}
\usepackage{amsthm}
\usepackage{amsfonts}
\usepackage{amssymb}
\usepackage{eucal}
\usepackage{yfonts}
\usepackage[all]{xy}
\usepackage{amsxtra}
\usepackage{appendix} 
\usepackage{tensor} 

\usepackage[pdftex, breaklinks, linktocpage=true, bookmarksopen=true,bookmarksopenlevel=0,bookmarksnumbered=true]{hyperref}

\hypersetup{
pdfauthor = {\textcopyright\ Tim holland},
pdfcreator = {\LaTeX\ with package \flqq hyperref\frqq},
colorlinks = {true}
}

\DeclareMathOperator{\im}{im}
\DeclareMathOperator{\rk}{rk}
\DeclareMathOperator{\surj}{surj}

\DeclareMathOperator{\bij}{bij}
\DeclareMathOperator{\nil}{nil}
\DeclareMathOperator{\End}{End}
\DeclareMathOperator{\id}{id}
\DeclareMathOperator{\GL}{GL}
\DeclareMathOperator{\Gr}{Gr}
\DeclareMathOperator{\Hom}{Hom}
\DeclareMathOperator{\Span}{span}

\theoremstyle{plain}
  \newtheorem{theorem}{Theorem}
  
  \newtheorem{lemma}[theorem]{Lemma}

\theoremstyle{definition}
  \newtheorem{definition}[theorem]{Definition}

\theoremstyle{remark}

\include{header}

\numberwithin{theorem}{section}  
\numberwithin{equation}{section}

\begin{document}
\bibliographystyle{plain}     

\title{Counting Semilinear Endomorphisms Over Finite Fields}
\author{Tim Holland}
\date{\today}
\thanks{The author is grateful to David Zureick-Brown for 
sharing his ideas and to Bryden Cais and Jeremy Booher for supervising his 2011 PROMYS research project.
}

\maketitle

\section{Introduction}\label{intro} 

Fix a prime $p$ and finite field $k$ of order $q:=p^r$.  
For a field automorphism 
$\tau$ of $k$ and a $k$-vector space $V$ of dimension $g$,  we will write $\End^{\tau}(V)$
for the set of {\em $\tau$-semilinear endomorphisms of $V$}; that is, additive maps
$F:V\rightarrow V$ which satisfy $F(\alpha v) = \tau(\alpha)\cdot F(v)$
for all $v\in V$ and $\alpha\in k$. 
As $\tau$ is an automorphism, the kernel and image
(as sets, say) of any $\tau$-semilinear map $F$ are both $k$-subspaces of $V$, 
and this gives a well-defined notion of rank and nullity.
Moreover, one has a canonical direct sum decomposition $V \simeq V^{F-\bij}\oplus V^{F-\nil}$
where $V^{F-\bij}$ is the maximal subspace of $V$ on which $F$ is bijective,
so we may speak of the ``infinity rank" of $F$, which by definition is $\rk_{\infty}(F):=\dim_k(V^{F-\bij})$.
For any pair of nonnegative integers $r,s$ satisfying $s \le r\le g$, we may thus define the set
\begin{equation}
	P_{r,s}^{\tau}:=\{F\in \End^{\tau}(V)\ :\ \rk(F) = r\ \text{and}\ \rk_{\infty}(F)=s\}.\label{PrsDef}
\end{equation}
These sets show up naturally in the classification of finite flat $p$-power order group
schemes over $k$ which are killed by $p$ and which have $p$-rank $s$, via
Dieudonn\'e theory.  It is therefore natural to ask for a closed formula (in $q=\#k$)
for the cardinality of $P^{\tau}_{r,s}$, and the main result of this paper
is precisely such a formula:
\begin{theorem}\label{main}
	Let $g$ be a positive integer and $r,s$ nonnegative integers with $s\le r\le g$.
	For any fixed automorphism $\tau$ of $k$ and any $g$-dimensional vector space $V$,
	the number of $\tau$-semilinear endomorphisms of $V$ with rank $r$ and infinity rank $s$ is
	\begin{equation*}
		\#P_{r,s}^{\tau} = 
		\frac{q^{g^2}}{q^{(g-r)^2+r-s}}\frac{ \prod\limits_{j=1}^g (1-q^{-j})  \prod\limits_{j=g-r}^{g-s-1} (1-q^{-j})}	
		{\prod\limits_{j=1}^{r-s} (1-q^{-j}) \prod\limits_{j=1}^{g-r}(1-q^{-j})}
	\end{equation*}
\end{theorem}
Here, we follow the usual convention that a product indexed by the empty set takes
the value 1.
Note that $P_{g,g}^{\id}$ is identified with $\GL_g(k)$ upon choosing a basis of $V$,
while the union of $P_{r,0}^{\id}$ for $0\le r\le g$ is, upon choosing a basis of $V$, 
the set of nilpotent $g\times g$ matrices with entries in $k$, so our formula
may be used to recover the well-known formulae for the order of $\GL_g(k)$
and for the number of nilpotent $g\times g$ matrices over $k$ (for which, see \cite{Crabb}).
In fact, our argument is a natural generalization of the proof of the main result
of \cite{Crabb}, though some care is required in our method to deal with the issue of semilinearity.
We remark that Theorem \ref{main} provides key input for one of the main results of
\cite{CEZ}, and indeed this was the genesis of the present note.

\section{Flags and adapted bases}

In this section, we summarize the concepts and tools from semilinear algebra that will
figure in the proof of Theorem \ref{main}.  We keep the notation of \S\ref{intro};
in particular, a $g$-dimensional $k$-vector 
space $V$. 

We begin by noting that for an automorphism $\tau$ of $k$, the set $\End^{\tau}(V)$
is naturally a $k$-vector space, and that for $F\in \End^{\tau}(V)$
and $F'\in \End^{\tau'}(V)$, the composition $F\circ F'$ lies in $\End^{\tau\circ\tau'}(V)$.
For $F\in \End^{\tau}(V)$, one checks that the subsets
$\ker(F)$ and $\im(F)$, defined in the usual way, are actually $k$-linear subspaces of $V$,
and we set $\rk(F):=\dim_k(\im(F))$.  
By definition, the {\em terminal image} of $F$ is the subspace
\begin{equation*}
	V^{F-\bij}:= \bigcap_{n\ge 0} \im(F^n),
\end{equation*}
and we define the infinity rank of $F$ to be the dimension of its terminal
image: $\rk_{\infty}(F):=\dim_k V^{F-\bij}$.  An easy argument shows that
in fact $\rk_{\infty}(F)=\rk(F^g)$, and that $F$ is bijective on $V^{F-\bij}$,
which justifies the notation.  In fact, $V^{F-\bij}$ is the maximal $F$-stable
subspace of $V$ on which $F$ is bijective \cite{}.

\begin{definition}
	Let $r$ be a nonnegative integer and
	\begin{equation}
		V=V_0\supsetneq V_1\supsetneq \cdots \supsetneq V_{r-1} \supsetneq V_r=0
			\label{flag}
	\end{equation}
	be a flag in $V$. Set $d_i= \dim V_i$.
	We say that an ordered basis $\{v_1,v_2,\ldots, v_g\}$ of $V$ is {\em adapted}
	to the given flag (\ref{flag}) if $\{v_{g-d_i+1},\ldots, v_n\}$ is a basis 
	of $V_i$ for all $i$.
\end{definition}

Given a flag in $V$ and a fixed ordered basis $\mathbf{e}:=\{e_i\}_{1\le i\le g}$
of $V$, there is a canonical procedure from the theory of Schubert cells
which associates to $\mathbf{e}$ an adapted basis
of the given flag, which we now explain.  

First, suppose that $U$ is an arbitrary 
subspace of $V$, and for $1\le j\le g$ define
\begin{equation*}
	U_j:= U \cap \Span\{e_{j+1},e_{j+2},\ldots, e_g\},
\end{equation*}
with the convention that $U_g=0$.  We then have descending chain of subspaces
\begin{equation*}
	U=U_0 \supseteq U_1 \supseteq \cdots \supseteq U_g=0
\end{equation*}
with the property that $\dim(U_{j-1})-\dim(U_{j})\le 1$
and equality holds if and only if $U_{j-1} \neq U_{j}$.  We define
\begin{equation*}
	J:=\{j\ :\ U_{j-1} \neq U_{j}\}
\end{equation*}
and put $m:=\#J$; note that this integer is equal to the dimension of $U$.

\begin{lemma}\label{adapt}
	For each $j\in J$, there is a unique vector $u_j\in U_{j-1}$ with 
	\begin{equation*}
		u_j-e_j\in \Span\{e_i\ :\ i>j,\ i\not\in J\}.\label{adaptcondition}
	\end{equation*}
	Moreover, $\{u_j\ :\ j\in J\}$ is a basis of $U$, and if
	the last $n$-vectors of $\mathbf{e}$
	lie in $U$ for some $n\le m$, then $u_j=e_j$
	for all $j\in J$ with $j \ge g-n+1$.
\end{lemma}

\begin{proof}
	We list the $m$ elements of $J$ in increasing order $j_1<j_2<\cdots <j_m$. 
	By definition of $J$, the complement of $U_{j_i}$ in $U_{j_i-1}$ is 1-dimensional
	for $1\le i\le m$, so we may pick a nonzero vector $v_{j_i}$ in this complement which spans it.
	Then $v_{j_m}$ is a linear combination of $\{e_{i}\ :\ i \ge j_m\}$     
	with a nonzero coefficient of $e_{j_m}$ by construction.  We may therefore uniquely
	scale $v_{j_m}$ by the inverse of this coefficient to obtain a vector $u_{j_m}\in U_{j_m}$
	which satisfies (\ref{adaptcondition}).
	Now suppose inductively that vectors $u_{j_m},u_{j_{m-1}},\ldots, u_{j_{d+1}}$ satisfying the
	condition of the Lemma have been uniquely determined.  We may uniquely write our choice $v_{j_d}$
	as
	\begin{equation*}
		v_{j_d} =  c_0 e_{j_d} + \sum_{i=1}^{g-j_d} c_i e_{j_d+i}
	\end{equation*}
	with $c_0$ necessarily nonzero.  We then define
	\begin{equation*}
		u_{j_d} = c_0^{-1} v_{j_d} - \sum_{\substack{j_d+i\in J\\ 1\le i\le g-j_d}} c_0^{-1}c_i e_{j_d+i}
	\end{equation*}
	One checks that $u_{j_d}$ satisfies (\ref{adaptcondition}).
	Multiplying our choice $v_{j_d}$ by any nonzero scalar gives the same
	vector $u_{j_d}$; in particular, the $u_{j_i}$ are independent
	of our initial choices of the $v_{j_i}$ and so are uniquely determined.
	That the set $\{u_j\ :\ j\in J\}$ is a basis of $U$ follows immediately from our construction,
	as does the fact that $u_j=e_j$ for $j\ge g-n+1$ when the last $n$ vectors of
	the ordered basis $\mathbf{e}$ lie in $U$.
\end{proof}

For a fixed ordered basis $\mathbf{e}=\{e_i\}$
of $V$ and a subspace $U$ of $V$, the procedure of Lemma \ref{adapt}
yields, in a canonical way, a new ordered basis 
$\{e_j\ :\ j\not\in J\}\cup \{u_j\ :\ j\in J\}$ of $V$ with the property that the final $m$
vectors are a basis of $U$.  We will say that this process {\em adapts} the basis $\mathbf{e}$
to the subspace $U$.  Note that the process of adapting an ordered basis to $U$ does not change the
final $n$ vectors when these vectors lie in $U$.  

Given a flag (\ref{flag}) in $V$ and a fixed ordered basis $\mathbf{e}$
of $V$, we now associate a canonical adapted basis $\mathbf{v}$ as follows.
First, we adapt $\mathbf{e}_r:=\mathbf{e}$ to $V_{r-1}$ to obtain a new ordered basis
$\mathbf{e}_{r-1}$ of $V$ with the property that the last $d_{r-1}$ vectors are a basis
of $V_{r-1}$.  We then adapt $\mathbf{e}_{r-1}$ to $V_{r-2}$ to obtain a new ordered
basis of $V$ in which the last $d_{r-2}$ vectors are a basis of $V_{r-2}$
and the last $d_{r-1}$ vectors are a basis of $V_{r-1}$ (as $V_{r-1}\subseteq V_{r-2}$
so the last $d_{r-1}$ vectors of $\mathbf{e}_{r-2}$ and $\mathbf{e}_{r-1}$
coincide as we have noted).  We continue in this manner, until we arrive
at the adapted basis $\mathbf{v}:=\mathbf{e}_1$; by the unicity of Lemma \ref{adapt},
this $\mathbf{v}$ is uniquely determined by the flag \ref{flag} and the fixed ordered basis $\mathbf{e}$
of $V$.

\section{Proof of Theorem \ref{main}}

Our proof of Theorem \ref{main} will proceed in two steps.
First, using flags and adapted bases,
we will show that the set $P_{r,s}^{\tau}$ defined by (\ref{PrsDef})
is in bijection with a certain set consisting of lists of vectors;
we will then count this latter set.

\begin{definition}\label{Xdef}
	For $r,s$ nonnegative integers with $s\le r\le g$, we define $X_{r,s}$ to be the
	subset of $V^g$ consisting of all $g$-tuples $(x_1,x_2,\ldots, x_g)$ which satisfy:
	\begin{enumerate}
		\item $\dim \Span\{x_j\}_{j=1}^g = r$
		\item $\dim \Span\{x_j\}_{j=g-s+1}^g = s$
		\item $x_{g-s}\in \Span \{x_j\}_{j=g-s+1}^g$.\label{third} 
	\end{enumerate}
\end{definition}

Now let $\tau$ be any automorphism of $k$ and fix, once and for all, a choice $\mathbf{e}$ of $k$-basis of $V$.
Any $F\in P_{r,s}^{\tau}$ determines a flag in $V$ via $V_i:=F^i(V)$, and this flag
necessarily has the form
\begin{equation}
	V= V_0 \supsetneq V_1\supsetneq V_2 \supsetneq \cdots \supsetneq V_t =V_{t+1} = V_{t+2}\cdots\label{Fflag}
\end{equation}
with $V_1=\im(F)$ of dimension $r$ and $V_t$ of dimension $s$ (as it is the terminal image of $F$).
Adapting $\mathbf{e}$ to (\ref{Fflag}) as in \S\ref{adapt} uniquely determines an ordered basis 
$\mathbf{v}_F:=\{v_{F,i}\}_{i=1}^g$
of $V$ (that depends on $F$), and we define a map of sets
\begin{equation}
		\xymatrix{
			{\mu:P_{r,s}^{\tau}}\ar[r] & {V^g}
			}  
			\qquad\text{by}\qquad
			\mu(F):=(F(v_{F,1}),F(v_{F,2}),\ldots, F(v_{F,g})).\label{mumap}
\end{equation}

The following Lemma is key:

\begin{lemma}\label{mubij}
	The map $\mu$ of $(\ref{mumap})$ is a bijection onto $X_{r,s}$.
\end{lemma}

\begin{proof}
	It is clear from our construction that $\mu$ has image contained in $X_{r,s}$, so it suffices
	to construct an inverse mapping which we do as follows.
	Given an arbitrary element $x:=(x_i)_{i=1}^g$ of $X_{r,s}$, we set $V_0:=V$ and $d_0:=g$
	and for $i\ge 1$ inductively construct a flag in $V$ by defining
	\begin{equation}
		V_i:=\Span\{x_{g-d_{i-1},\ldots,x_g}\}\qquad\text{and}\qquad d_i:=\dim(V_i).\label{inductflag}
	\end{equation}
	Letting $\mathbf{v}_{x}=\{v_{x,i}\}_{i=1}^g$ be the adaptation of the basis $\mathbf{e}$ to this flag,
	we define $F_x\in \End^{\tau}(V)$ to be the unique $\tau$-semilinear endomorphism 
	of $V$ satisfying $F_x(v_{x,i})=x_i$ for all $i$.  That is, for arbitrary $v\in V$, we write $v=\sum c_i v_{x,i}$
	as a unique linear combination of the basis vectors $v_{x,i}$ and we define
	$F_x(v):=\sum \tau(c_i)x_i$, which is visibly a $\tau$-semilinear endomorphism of $V$.   
	We claim that $V_i=F_x^i(V)$ for all $i$.  For $i=0$ this is simply the definition of $V_0=V$.
	Inductively supposing that $F_x^i(V)=V_{i}$
	for some $i\ge 0$, we then have
	\begin{align*}
		F_x^{i+1}(V) = F_x(F_x^i(V)) = F_x(V_i) &= F_x(\Span\{v_{x,j}\}_{j=g-d_i+1}^g) \\
		&=\Span\{F_x(v_{x,j})\}_{j=g-d_i+1}^g \\
		&= \Span\{x_{j}\}_{j=g-d_i+1}^g \\ 
		&= V_{i+1},
	\end{align*}
	where in the final equality on the first line we have use the fact that the last $d_i$
	vectors in the adapted basis $\mathbf{v}_x$ span $V_{i}$ by construction.
	We conclude from the definition of $X_{r,s}$ that $F_x$ lies in $P_{r,s}^{\tau}$,
	and we define $\nu:X_{r,s}\rightarrow P_{r,s}(V)$ to be the map of sets which sends $x$ to $F_x$.  
	It is a then straightforward exercise
	to check that $\nu$ and $\mu$ are inverse mappings of sets.
\end{proof}

We now wish to enumerate the set $X_{r,s}$.  We remark that the set $X_{r,s}$
is independent of $\tau$, so that the semilinearity aspect of our
counting problem has been entirely removed at this point.
The $s$ vectors $x_{g-s+1},\ldots,x_g$ of $V$ must be linearly independent, but otherwise
may be chosen arbitrarily from $V$; in particular, there are 
\begin{equation}
		\prod_{i=0}^{s-1} (q^g - q^i)\label{Lastscount}
\end{equation}
ways to do this.  Supposing that $x_{g-s+1},\ldots, x_g$ have been chosen, 
we write $V_{\infty}$ for their span and we put $n:=g-s$ and $d:=r-s$. 
A choice $x_1,\ldots, x_n$ of $n$-vectors in $V$ will have the property
that the $g$-vectors $x_1,\ldots, x_g$ span an $r$-dimensional subspace
of $V$ if and only if the images of $x_1,\ldots,x_n$ span a $d$-dimensional
subspace of $W:=V/V_{\infty}$.  The condition (\ref{third}) in the definition
of $X_{r,s}$ (Definition \ref{Xdef}) that $x_{g-s}$ lie in $V_{\infty}$
is of course equivalent to the condition that its image in $W$ be zero.
We are therefore reduced to computing the cardinality of the set
\begin{equation*}
	Q_{n,d}:=\{(w_i)_{i=1}^{n-1}\in W^{n-1}\ :\ \dim\Span\{w_i\}_{i=1}^{n-1} = d\}. 
\end{equation*}
As usual, we write $\Gr(W,d)$ for the Grassmannian of $d$-dimensional subspaces 
of $W$ and for any $k$-vector space $U$, we denote by $\Hom_k^{\surj}(k^{n-1},U)$ 
the set of $k$-linear surjective homomorphisms from $k^{n-1}$ onto $U$.
We then define the set
\begin{equation*}
	A_{n,d}:=\{(U,T)\ :\ U\in \Gr(W,d)\ \text{and}\ T\in \Hom_k^{\surj}(k^{n-1},U)\}
\end{equation*}
as well as a map of sets
\begin{equation}
	\xymatrix{
		{\gamma:Q_{n,d}} \ar[r] & A_{n,d}
	}\qquad\text{by}\qquad \gamma((w_i)_{i=1}^{n-1}):=
	(\Span\{w_i\}_{i=1}^{n-1},\ T: \sum c_if_i \mapsto \sum c_iw_i)\label{gammamap}
\end{equation}
where $\{f_i\}_{i=1}^{n-1}$ is the standard basis of $k^{n-1}$.

\begin{lemma}\label{gammabij}
	The map $\gamma$ of $(\ref{gammamap})$ is bijective.
\end{lemma}

\begin{proof}
	The map $\delta:A_{n,d}\rightarrow Q_{n,d}$ sending $(U,T)$ to $\{Tf_i\}_{i=1}^{n-1}$ is 
	clearly inverse to $\gamma$.
\end{proof}

To count $Q_{n,d}$, it therefore suffices to count $A_{n,d}$.  To do this, 
we note that for any $d$-dimensional $k$-vector space $U$, choosing a basis of $U$
gives a bijection between the set 
$\Hom_{k}^{\surj}(k^{n-1},U)$ and the set of $(n-1)\times d$ matrices over $k$ with rank 
$d$, and we deduce that
\begin{equation}
		\#\Hom_{k}^{\surj}(k^{n-1},U) = \prod_{i=0}^{d-1} (q^{n-1} - q^i)
		\label{Homcount}
\end{equation}
for any such $U$ (note, in particular, that this is independent of $U$).  As the count
\begin{equation}
	\#\Gr(W,d) = \frac{\prod\limits_{i=0}^{d-1}(q^n - q^i)}{\prod\limits_{i=0}^{d-1} (q^d - q^i)}
	\label{Grcount}
\end{equation}
is standard, we conclude from (\ref{Grcount}), (\ref{Homcount}) and Lemma \ref{gammabij}
that 
\begin{equation}
	\#Q_{n,d} = \#\Gr(W,d)\cdot \#\Hom_{k}^{\surj}(k^{n-1},U) = 
	\frac{\prod\limits_{i=0}^{d-1}(q^n - q^i)}{\prod\limits_{i=0}^{d-1} (q^d - q^i)}
	\prod_{i=0}^{d-1} (q^{n-1} - q^i)\label{Qcount}.
\end{equation}

For each $w\in W=V/V_{\infty}$, there are $\# V_{\infty} = q^s$ ways to lift $w$ to a vector
in $V$, and hence $q^{sn}=q^{s(g-s)}$ ways to lift any list of $n=g-s$ vectors in $W$ to 
$V$.  Thus, by (\ref{Qcount}), the number of choices for the first $g-s$ vectors $(x_i)_{i=1}^{g-s}$
is
\begin{equation}
	q^{s(g-s)}\frac{\prod\limits_{i=0}^{r-s-1}(q^{g-s} - q^i)}{\prod\limits_{i=0}^{r-s-1} (q^{r-s} - q^i)}
	\prod_{i=0}^{r-s-1} (q^{g-s-1} - q^i)\label{Firstncount}.
\end{equation}
Combining (\ref{Firstncount}) and (\ref{Lastscount}) and using Lemma \ref{mubij} then gives
\begin{equation}
	\#P_{r,s}^{\tau} = \#X_{r,s} = 
	q^{s(g-s)}\frac{\prod\limits_{i=0}^{r-s-1}(q^{g-s} - q^i)}{\prod\limits_{i=0}^{r-s-1} (q^{r-s} - q^i)}
	\prod_{i=0}^{r-s-1} (q^{g-s-1} - q^i) \prod_{i=0}^{s-1} (q^g - q^i),
\end{equation}
which, after some elementary algebraic manipulation, is readily seen to be equivalent to 
the formula of Theorem \ref{main}.


\bibliography{cseofqbib}

\end{document}